\documentclass{article}

\usepackage{authblk}

\usepackage{amsthm} 
\usepackage{amssymb}
\usepackage[mathscr]{euscript}
\usepackage{amsmath}

\usepackage{enumitem}
\setlist[enumerate,1]{label=(\roman*)}

\newcommand{\Cay}[2]{{\rm Cay}(#1,#2)}
\newcommand{\F}{\mathbb{F}_2}
\newcommand{\scrC}{\mathscr{C}}
\newcommand{\vspan}[1]{{\rm rowspan}(#1)}

\newcommand{\ann}[1]{{\rm Ann}(#1)}
\renewcommand{\it}{\textit}
\renewcommand{\ll}{\langle}
\newcommand{\rr}{\rangle}

\newcommand{\Null}[1]{{\rm Null}(#1)}

\usepackage{indentfirst}
\usepackage{geometry}

\theoremstyle{plain}
\newtheorem{theorem}{Theorem}[section]
\newtheorem{corollary}[theorem]{Corollary}
\newtheorem{lemma}[theorem]{Lemma}
\newtheorem{proposition}[theorem]{Proposition}
\newtheorem{definition}{Definition}[section]

\usepackage{hyperref}
\hypersetup{hypertex = true,
colorlinks = true,
linkcolor = blue,
anchorcolor = blue,
citecolor = blue}

\title{\MakeUppercase{\large construction of storage codes of rate approaching one on triangle-free graphs}}
\author{Hexiang Huang\qquad   Qing Xiang\thanks {Research partially supported by the National Natural Science Foundation of China grant 12071206, 12131011, 12150710510, and the Sino-German Mobility Programme M-0157}
}

\date{December 29, 2022}

\begin{document}

\maketitle

\begin{abstract}
    Consider an assignment of bits to the vertices of a connected graph $\Gamma(V, E)$ with the property that the value of each vertex is a function of the values of its neighbors. A collection of such assignments is called a storage code of length $|V|$ on $\Gamma$. In this paper we construct an infinite family of binary linear storage codes on triangle-free graphs with rates arbitrarily close to one. 
\end{abstract}

$\textbf{\small\it{Keywords.}}$
Storage code, Cayley graph, Group algebra, Annihilator

\section{Introduction}

Let $\Gamma$ be a simple graph with $n$ vertices and $\scrC$ a binary code of length $n$ with coordinates indexed by the vertices of $\Gamma$. We say that $\scrC$ is a $\it{storage code}$ on $\Gamma$ if for each codeword $c \in \scrC$, every coordinate of $c$ can be uniquely determined by the values of its neighbors. 

Given a graph $\Gamma$, we can define a storage code on it: Assume $\Gamma$ has $n$ vertices, say $v_1,v_2,\ldots,v_n$. Let $A(\Gamma)$ be the adjacency matrix of $\Gamma$ with the rows and columns labeled by $v_1,v_2,\ldots,v_n$. Let $H = A(\Gamma)+I$ and $\scrC$ be the linear code over $\F$ with parity-check matrix $H$. Let $c=(c_1,c_2,\ldots,c_n)\in \scrC$ be a codeword. For the $v_i^{\rm th}$ entry of $c$, namely $c_i$, row $v_i$ of $H$ gives rise to a linear equation, $c_i=\sum_{v_j \in N(v_i)}c_j$, where $N(v_i)$ is the set of neighbors of $v_i$ in $\Gamma$; hence $c_i$ can be recovered by the values of its neighbors. Storage codes on graphs were introduced by Mazumdar \cite{Mazumdar2014StorageCO, Mazumdar2017StorageCA} and Shanmugam and Dimakis \cite{Shanmugam2014BoundingMU}. The authors of \cite{cameron2014guessing} and \cite{chris2011guessing} also introduced the concept of storage codes on graphs, although in a different guise. We will only consider binary linear storage codes. The rate of a storage code $\scrC$, denoted by $R(\scrC)$, is simply the ratio of its dimension to the dimension of the ambient space. If we have a family of storage codes $\scrC_r$, where $r$ is a parameter, assuming $\lim_{r\to \infty}R(\scrC_r)$ exists, then this limit is called the rate of the family.

It is easy to construct a storage code of rate close to one: Let $\Gamma$ be the  complete graph on $n$ vertices, and $\scrC$ the binary linear code defined by a single parity-check equation $\sum_{i=1}^{n}x_i=0$. Then $\scrC$ is a storage code on $\Gamma$ and of rate $1-1/n$.

In the above example, the graph used to obtain the storage code of rate close to one is very dense (in fact, as dense as possible), and contains a large number of cliques. It is therefore natural to consider the question of the largest attainable rate of storage codes on graphs that contain no cliques, i.e., triangle-free graphs.

Constructing storage codes of high rate on such graphs represents a challenge. Consider a densest triangle-free graph, namely a complete bipartite graph $K_{m,n}$ and a storage code $\scrC$ on it. There are two independent vertex sets of $K_{m,n}$. For each vertex, we can recover the message on it from the ones in the other (vertex) independent set, so $R(\scrC)\leq 1/2$. In early studies \cite{chris2011guessing}, the authors had conjectured that for triangle-free graphs, $R = 1/2$ is the largest attainable rate value. Later on this conjecture was refuted in \cite{cameron2014guessing} by some isolated examples.


Recently Barg and Zémor \cite{barg2022high} presented four infinite families of storage codes on triangle-free graphs. They used the \textit{Cayley graph method}: Let $S$ be a subset of $\F^r$ such that $0\notin S$ and any three vectors in $S$ are linearly independent. Then $\Gamma = \Cay{\F ^r}{S}$ is triangle-free. Let $H = A(\Gamma) + I$ and  the binary linear code $\scrC$ be defined by using $H$ as a parity-check matrix. Then we obtain a storage code $\scrC$ on the triangle-free graph $\Gamma$. By this method, a subset $S\subseteq \F^r$ will give rise to a triangle-free graph and a storage code on it. 

\vspace{0.1in}

\noindent\textbf{Barg and Zémor's Constructions} \cite{barg2022high}
\begin{itemize}
    \item \textit{The Clebsch family:} This is an infinite family of storage codes on generalized Clebsch graphs having rates above and decreasing to $1/2$.
    
    \item \textbf{\textit{The Hamming family:}} To construct this family, start with the parity-check matrix of the extended binary Hamming code of length $2^{r-1}$, augment it with an all-zero column, and then add a row of weight $2$ with the first entry and the last entry being 1. Denote the resulting $(r+1)\times n$ matrix by $M_r$. For instance, for $r=4$ we obtain the matrix
    \begin{equation*}
        M_4 = \left[ \begin{array}{cccccccc|c}
         0 & 0 & 0 & 0 & 1 & 1 & 1 & 1 & 0 \\
         0 & 0 & 1 & 1 & 0 & 0 & 1 & 1 & 0 \\
         0 & 1 & 0 & 1 & 0 & 1 & 0 & 1 & 0 \\
         1 & 1 & 1 & 1 & 1 & 1 & 1 & 1 & 0 \\
         \hline
         1 & 0 & 0 & 0 & 0 & 0 & 0 & 0 & 1
         \end{array}  \right].
    \end{equation*}
    Let $S$ be the set of column vectors of $M_r$. The rates of this family of codes are increasing to $3/4$. We will give a new proof for this fact in Theorem \ref{hammingRate} and then generalize this family.
    
    \item \textit{The BCH family:} Let $B_s$ be a parity-check matrix of the BCH code and $S$ the column set of $B_s$. It remains an open problem that whether the rate of this family could reach one. But by using a computer, it is shown that the rate increases as $s$ varies from $4$ to $8$, especially when $s=8$ the rate is $0.8196$, which represents the largest known rate of storage codes on triangle-free graphs at the time of writing of \cite{barg2022high}.

    \item \textit{The Reed-Muller family:} Let $M$ be a matrix composed of rows which are the evaluation vectors of linear and quadratic Boolean polynomial functions. Let $S$ be the column set of $M$. It is a family perfectly satisfying the necessary conditions of having high rate; but we do not know the rate of this family.
    
\end{itemize}

The rates of first two families were computed by using matrix method \cite{barg2022high}. In addition to the constructions, the authors of \cite{barg2022high} also formulated necessary conditions for the code rate to be high. 

\begin{theorem}[\cite{barg2022high}]
    \label{NeceCon}
    Let $\Gamma = \Cay{\F ^r}{S}$ with $S$ not containing the zero vector. Let $M$ be the $r \times n$ matrix whose columns are made up of the elements of $S$ and $\mathscr{C}$ the binary code with parity-check matrix $H=A(\Gamma)+I$. Then
    \begin{enumerate}
        \item If $R(\mathscr{C})>1/2$, then $n$ is odd and all the rows of $M$ have even weight;
        \item If $R(\mathscr{C})>(2^k-1)/2^k, k = 2,3,\ldots$ , then any $k$ rows of $M$ have even intersection.
    \end{enumerate}
    
\end{theorem}
\section{A family of storage codes on triangle-free graphs of rate one}
\label{ConstructionSection}

\begin{definition}
    Let $G$ be a finite multiplicatively  written group with identity element $e$, and let $S$ be a subset of $G$ such that $e\notin S$ and $S=S^{-1}$, where $S^{-1}=\left\{g^{-1}\ \vert\ g\in S\right\}$. The Cayley graph on $G$ with connection set $S$, denoted by $\Gamma=\Cay{G}{S}$,  is the graph with elements of $G$ as vertices, two vertices $g_1,g_2\in G$ are adjacent if and only if $g_1g_2^{-1}\in S$.
\end{definition}

In this paper we will only use the Cayley graph method to construct storage codes. Recall some notation: $S$ is a subset of $\F^r$ with $0\notin S$, $\Gamma = \Cay{\F^r}{S}$ and $H = A(\Gamma) + I$. The storage code $\scrC$ is a linear code over $\F$ with parity-check matrix $H$. Note that in $\Gamma$, each vertex represents a vector in $\F^r$ and the rows and columns of $H$ are labeled by the vectors of $\F^r$. If there is a triangle in $\Gamma$, i.e., there are three vectors $v_1,v_2,v_3$ such that $v_1-v_2,v_2-v_3,v_3-v_1 \in S$, then there are three vectors $w_1,w_2,w_3 \in S$ such that $w_1+w_2+w_3=0$. So if the sum of any three vectors in $S$ is not zero, then $\Gamma$ is triangle-free.

If we have an $S\subseteq \F^r$, by definition of Cayley graph we have
\begin{equation*}
    A(\Gamma) = (a_{uv}),\quad a_{uv}=\,
    \left\{
    \begin{aligned}
        &1,\ if\,\, v-u \in S\, ,\\
        &0,\ otherwise.
    \end{aligned}
    \right.
\end{equation*}
As $H = A(\Gamma) + I$, we have
\begin{equation}
    H = (h_{uv}),\quad h_{uv}=\,
    \left\{
    \begin{aligned}
        &1,\ if\,\, v-u \in S',\label{minus=plus}\\
        &0,\ otherwise,
    \end{aligned}
    \right.
\end{equation}
where $S'=S\cup \left\{0\right\}$. We say that $H$ is the \it{coset matrix} of $S$. It is more convenient to consider $S$ containing the zero vector. Since $H$ is a parity-check matrix of $\scrC$, $\scrC$ is the set of row vectors $c$ which satisfy $H\cdot c^\intercal = 0$. Note that in (\ref{minus=plus}), $v-u = v+u$ as we are working over $\F$, so $H$ is symmetric. Thus $\scrC = \left\{c\in \F^{2^r}\,\vert\,c\cdot H=0\right\}$, we write $\scrC=\Null{H}$, the null space of $H$.

Our purpose is to construct an $S\subseteq \F^r$ which satisfies $0 \in S$ and the sum of any three nonzero vectors in $S$ is not zero, such that the rate $R(\Null{H})$ is high. We find that if we translate $S$ in the Hamming family into a more algebraic setting, then it will suddenly become concise and easy to generalize.

The ambient space of $\scrC$ is $\F^{2^r}$, by using $c=(c_v\,\vert\,v \in \F^r)$ and $\sum_{v\in \F^r} c_vv$ interchangeably, we can view $\F^{2^r}$ as the group algebra $\F[\F^r]$. Denote the addition in the vector space $\F^r$ by ``$\oplus$'', then $\F[\F^r]$ is a ring with addition ``$+$" and multiplication ``$\oplus$". 

Note that each row of $H$ is the characteristic vector of some subset of $\F^r$, by abusing these two notions, we observe that the row $v$ of $H$ is the coset $v\oplus S$ in $\left(\F ^r,\oplus\right)$. Just as cyclic codes can be treated as principal ideals in a certain group ring, we will show that $\vspan{H} \subseteq \F[\F^r]$ is isomorphic to a principal ideal in the corresponding ring and the null space $\scrC = \Null{H}\subseteq \F[\F^r]$ is isomorphic to the annihilator of the principal ideal.
 
Let $G = \langle x_1,x_2,\ldots,x_r|x_1^2=x_2^2=\cdots =x_r^2=1\rangle$ be the elementary abelian 2-group of size $2^r$. To avoid the trouble of two addition symbols, we will use the following group isomorphism
\begin{equation*}
\begin{aligned}
    \sigma :\left(\F ^r,\oplus\right) &\rightarrow (G, \cdot )\  \\
            (a_1,a_2,\ldots ,a_r) &\mapsto x_1^{a_1}x_2^{a_2}\cdots x_r^{a_r}\ 
\end{aligned}
\end{equation*}
Note that $\sigma$ induces a ring isomorphism between group algebras 
\begin{equation*}
\begin{aligned}
    \phi :(\F[\F^r],+,\oplus) &\rightarrow (\F[G],+,\cdot) \simeq P_r  \\
            \sum_{v\in \F ^r}k_v v &\mapsto \sum_{v\in \F ^r}k_v  \sigma\left(v\right) 
\end{aligned}
\end{equation*}
where $P_r := {\F \left[x_1,x_2,\ldots ,x_r\right]}/{\langle x_1^2-1,x_2^2-1,\ldots ,x_r^2-1\rangle}$.

 Let $z$ be an element in some commutative ring $A$. Recall that the principal ideal $\ll z \rr = \left\{a\cdot z\ \vert\ a\in A \right\}$ and the annihilator $\ann{z} = \left\{ a \in A\ \vert\ a\cdot z = 0\right\}$. 
 
 Assume that $S\subseteq \F^r$ contains the zero vector. Let $f_S=\sum_{w\in S}\sigma(w)\in P_r$ and $H$ be the coset matrix of $S$. We then have the following two propositions.

\begin{proposition}
    The dual code $\scrC ^\perp = \vspan{H}$ is isomorphic to the principal ideal $\ll f_S\rr \subseteq P_r$.
\end{proposition}

\begin{proof}
    Denote the row indexed by $v\oplus S$ in $H$ by $h_v$, $v\in \F ^r$. Note that $h_v$ is a vector in $\F^{2^r}$, and so an element in $\F[\F^r]$; also $h_0 = \sum_{w\in S}{w}$ and $ h_v = \sum_{w \in S}{v\oplus w} = v\oplus h_0$.
    Hence
    \begin{equation*}
        \begin{aligned}
        \vspan{H} &= \left\{\sum_{ v \in \F^r}{k_v h_v } \,\vert \,k_v \in \F \right\}
        = \left\{\sum_{v \in \F^r}{k_v \left(v\oplus h_0\right)} \,\vert \, k_v \in \F \right\} \\
        &= \left\{\left(\sum_{v \in \F^r}{k_v v}\right)\oplus h_0 \,\vert \,k_v \in \F \right\}
        = \left\{ a\oplus h_0\,\vert \,a \in \F[\F^r] \right\}.
     \end{aligned}
    \end{equation*}

    For each $a\oplus h_0 \in \vspan{H}$, we define
    \begin{equation*}
    \begin{aligned}
        \phi'\left(a\oplus h_0\right) &= \phi\left(a\oplus h_0\right) = \phi\left(a\right) \cdot \phi\left(h_0\right) \\
        &= \phi\left(a\right) \cdot \phi\left(\sum_{w\in S}{w}\right) = \phi\left(a\right) \cdot \sum_{w \in S}{\sigma\left(w\right)}=\phi(a)\cdot f_S.
    \end{aligned}
    \end{equation*}
As $\phi'(a\oplus h_0) \in \ll f_S\rr$, the map $\phi'$ is well-defined. Furthermore the map $\phi'$ is linear and injective since it is the restriction of $\phi$ to $\scrC^\perp $. It is also surjective: $\phi$ is bijective, so it has an inverse. Then for each $g\in P_r$, we have $\phi'(\phi^{-1}(g)\oplus h_0)=\phi(\phi^{-1}(g))\cdot f_S=g\cdot f_S$, so $g\cdot f_S$ has an inverse image $\phi^{-1}(g)\oplus h_0 \in \vspan{H}$. We conclude that $\phi'$ is a linear bijection.
    
\end{proof}

\begin{proposition}
    The code $\scrC = \Null{H}$ is isomorphic to $\ann{f_S} \subseteq P_r$.
\end{proposition}

\begin{proof}
    For each $a \in \Null{H}$, we define $\phi''(a) = \phi(a)$. Note that
    \begin{equation*}
    \begin{aligned}
        \Null{H} &= \left\{ a \in \F^{2^r} \,\vert \, \sum_{v \in \F^r}{a_v  h_v} = 0 \right\} 
        = \left\{ a \in \F^{2^r} \,\vert \, \sum_{v \in \F^r}{a_v \left(v\oplus h_0\right)} = 0\right\} \\
        &= \left\{ a \in\F^{2^r} \,\vert \, \left(\sum_{v \in \F^r}{a_v v}\right) \oplus h_0 = 0 \right\} 
        = \left\{ a \in\F[\F^r] \,\vert \, a \oplus h_0 = 0 \right\}.
    \end{aligned}
    \end{equation*}
    Note that in the last equality, we change from the linear space $\F^{2^r}$ to the group algebra $\F[\F^r]$.
    
    We can verify that $\phi''$ is well-defined: For each $a \in \Null{H},\ \phi''(a)\cdot f_S = \phi(a) \cdot \phi(h_0) = \phi(a\oplus h_0) = 0$ and thus $\phi''(a) \in \ann{f_S}$. $\phi''$ is linear and injective as it is the restriction of $\phi$ to $\Null{H}$. It is also surjective: For each $g = \sum_{m\in G}{k_mm} \in \ann{f_S}$, we can find the inverse image $\sum_{m\in G}{k_m\sigma^{-1}(m)}$ in $\Null{H}$. Hence $\phi''$ is a linear bijection.
\end{proof}

Now we can work in the quotient ring $P_r = {\F \left[x_1,x_2,\ldots ,x_r\right]}/{\langle x_1^2-1,x_2^2-1,\ldots ,x_r^2-1\rangle}$. Assume that $S\subseteq \F^r$ contains the zero vector and the sum of any three nonzero vectors in $S$ is not zero. We will usually view a storage code as the annihilator of $f_S\in P_r$ without explicitly referring to the ambient graph $\Cay{\F^r}{S\setminus \left\{0\right\}}$ . Below we list some properties of $P_r$:
\begin{enumerate}
    \item\label{Property1} $\left(x_i+1\right)^2 = 0$ in $P_r$, for $i = 1,\ldots ,r$ ;
    \item\label{Property2} $R\left( \ll x_1 + 1 \rr + \ll x_2 + 1\rr + \cdots + \ll x_k + 1\rr \right) = 1 - \frac{1}{2^k} $ ;
    \item\label{Property3} $R \left(\prod_{i = 1}^{n}{\sum_{j = 1}^{k_i}{\ll x_{ij}+1\rr}}\right) 
    =\prod_{i = 1}^{n}{R\left(\sum_{j = 1}^{k_i}{\ll x_{ij}+1\rr}\right)}$ .    
\end{enumerate}
    
To prove the first property, observe that $(x_i+1)^2=x_i^2+1$ is in the ideal ${\langle x_1^2-1,x_2^2-1,\ldots ,x_r^2-1\rangle}$, so it is zero in the quotient ring $P_r$. The proofs of the other two properties are given in Section \ref{SectionProof}; the second property will be proved in Theorem \ref{RateAdd} and the third in Corollary \ref{corollary}.

Using these properties, we can give a simple proof for the fact that the Hamming family is of rate $3/4$.

\begin{theorem}[The Hamming family]
    The storage codes in Hamming family are of the form $\ann{f_r}\subseteq P_{r+1}$, where 
    \[ f_r = \left(x_r+1\right)\left(x_{r+1}+1\right)+x_r\cdot \left(x_1 + 1\right)\left(x_2 + 1\right) \cdots \left(x_{r-1} +1\right).\]
    The ambient graph is triangle-free and the rate $R\left(\ann{f_r}\right)$ approaches $3/4$ as $r$ goes to infinity .
    \label{hammingRate}
\end{theorem}

\begin{proof}
    Let $M_r$ be the matrix constructed in the Hamming family. For example, when $r=4$, we have 
    \begin{equation*}
        M_4 = \left[ \begin{array}{cccccccc|c}
         0 & 0 & 0 & 0 & 1 & 1 & 1 & 1 & 0 \\
         0 & 0 & 1 & 1 & 0 & 0 & 1 & 1 & 0 \\
         0 & 1 & 0 & 1 & 0 & 1 & 0 & 1 & 0 \\
         1 & 1 & 1 & 1 & 1 & 1 & 1 & 1 & 0 \\
         \hline
         1 & 0 & 0 & 0 & 0 & 0 & 0 & 0 & 1
         \end{array}  \right]
    \end{equation*}
    It is clear that any three column vectors of $M_r$ are linearly independent, so the ambient graph is triangle-free. Remember that we need an $S\subseteq \F^{r+1}$ containing the zero vector, so we augment $M_r$ with a column of all-zeros to obtain $M_r'$, and let $f_r \in P_{r+1}$ be the element corresponding to $M_r'$. Then the storage code is $\ann{f_r}$. We have
    \begin{equation*}
    \begin{aligned}
        f_r &= x_{r+1}x_r + x_r\cdot \left(\sum_{\left(s_1,\ldots ,s_{r-1}\right)\neq \left(0,\ldots ,0\right)}{x_1^{s_1}x_2^{s_2}\cdots x_{r-1}^{s_{r-1}}}\right) + x_{r+1} + 1\\
        &= x_{r+1}x_r + x_r\cdot\left(\sum{x_1^{s_1}x_2^{s_2}\cdots x_{r-1}^{s_{r-1}}}\right) + x_r + x_{r+1} + 1 \\
        &= \left(x_r+1\right)\left(x_{r+1}+1\right)+x_r\cdot \left(x_1 + 1\right)\left(x_2 + 1\right) \cdots \left(x_{r-1} +1\right).
    \end{aligned}
    \end{equation*}
     Using Property \ref{Property1}, $x_i+1\in \ann{x_i+1}$, we then have an inclusion
    $$\ann{f_r} \supseteq \left(\ll x_r + 1\rr + \ll x_{r+1} +1\rr\right)\left(\ll x_1+1\rr+\ll x_2+1\rr+\cdots+\ll x_{r-1} +1\rr \right).$$
    Using Properties \ref{Property2} and \ref{Property3}, we obtain
    \begin{equation*}
    \begin{aligned}
        R\left(\ann{f_r}\right) &\geq R\left( \left(\ll x_r + 1\rr + \ll x_{r+1} +1\rr\right)\left(\ll x_1+1\rr + \ll x_2 + 1\rr+\cdots +\ll x_{r-1} +1\rr\right) \right) \\
        &= R\left( \ll x_r + 1\rr + \ll x_{r+1} +1\rr\right)\cdot R\left(\ll x_1+1\rr + \ll x_2 + 1\rr+\cdots +\ll x_{r-1} +1\rr\right) \\
        &= \frac{3}{4}\cdot \left(1-\frac{1}{2^{r-1}}\right).
    \end{aligned}
    \end{equation*}

    Observe that the last two rows of $M_r$ intersect in one position, so $R\left(\ann{f_r}\right) \leq 3/4$ by Theorem \ref{NeceCon}. Thus the limit of rates $R\left(\ann{f_r}\right)$ is $3/4$.
\end{proof}

\noindent\textbf{Remarks.}
We observe that the constructed element $f_r$ in Theorem \ref{hammingRate} is the sum of a ``short" term $(x_r+1)(x_{r+1}+1)$ and a ``long" term $x_r\cdot \left(x_1 + 1\right)\left(x_2 + 1\right) \cdots \left(x_{r-1} +1\right)$. Note that by Property \ref{Property2}
$$R(\ann{(x_r+1)(x_{r+1}+1)})\geq R(\ll x_r+1\rr +\ll x_{r+1}+1\rr) = 3/4.$$
So the ``short" term alone already achieves the rate of $3/4$. However there are monomials, $x_r,x_{r+1},x_rx_{r+1}$ (in the expansion of $(x_r+1)(x_{r+1}+1)$), which represent three linearly dependent vectors in $S$, so $(x_r+1)(x_{r+1}+1)$ will give rise to a Cayley graph with triangles. To eliminate the linear dependence, the ``long" term is needed, and if the number of additional variables, $r-1$, in the ``long" term is large then the rate value will not decrease much.

Consider $f=(x_1+1)(x_2+1)(x_3+1)\in P_3$. We have $\ann{f}\supseteq \ll x_1+1\rr + \ll x_2+1 \rr + \ll x_3+1\rr$ by Property \ref{Property1} and $R(\ann{f})\geq 7/8$ by Property \ref{Property2}. So it is easy to construct high-rate storage codes; however the ambient Cayley graph here is not triangle-free since there are many ``linearly dependent'' triples, such as $x_1,x_2,x_1x_2$.

\begin{theorem}
    \label{constrution}
    Let $\scrC_k = \ann{f_k} \subseteq P_{3k+3}$, where 
    \begin{equation*}
    \begin{aligned}
         f_k = &(x_{3k+1}+1)(x_{3k+2}+1)(x_{3k+3}+1)
         +x_{3k+1}x_{3k+2}\left(x_1+1\right)\cdots\left(x_k+1\right) \\
             &+x_{3k+1}x_{3k+3}\left(x_{k+1}+1\right)\cdots\left(x_{2k}+1\right)+x_{3k+2}x_{3k+3}\left(x_{2k+1}+1\right)\cdots\left(x_{3k}+1\right).
    \end{aligned}
    \end{equation*}
Then $\scrC_k$ is a storage code on a triangle-free graph. This family is of rate $7/8$. 
\end{theorem}

\begin{proof}
    Let $S \subseteq \F^{3k+3}$ be the subset of vectors corresponding to the monomials in the fully expanded $f_k$. We need to verify that $0\in S$ and the sum of any three nonzero vectors in $S$ is not zero. 
    
    There is a $1$ in the expansion of $f_k$, so $0\in S$. Note that the sum of any three nonzero vectors in $S$ is not zero is equivalent to the product of any three monomials in the expansion of $f_k+1$ is not $1\in P_{3k+3}$. There are several cases for the choices of three monomials:
    \begin{enumerate}
    
        \item Choose all three from the survived monomials of the ``short'' term $(x_{3k+1}+1)(x_{3k+2}+1)(x_{3k+3}+1)$. Note that the even-degree terms, $1,\,x_{3k+1}x_{3k+2},\,x_{3k+1}x_{3k+3},\,x_{3k+2}x_{3k+3}$, are canceled out in the expansion of $f_k+1$, so only odd-degree terms, namely, $x_{3k+1}$, $x_{3k+2}$, $x_{3k+3}$, and $x_{3k+1}x_{3k+2}x_{3k+3}$ survive, so the product of any three of them cannot be $1$.
        
        \item Choose all three from the survived monomial terms in the expansion of some ``long" term, for example, in the expansion of $x_{3k+1}x_{3k+2}(x_1+1)\cdots(x_k+1)$. Then the product of these three monomial terms will have the factor $x_{3k+1}^3x_{3k+2}^3=x_{3k+1}x_{3k+2}$, hence it is not equal to $1\in P_{3k+3}$. 

        \item There are three more cases: Choose two from the ``short'' term and one from some ``long'' term; one from the ``short'' term and two from ``long'' terms; three from more than one ``long'' terms. For all these cases, note that the product must have degree $1$ in some variable $x_i,i=1,\ldots,3k$, hence is not equal to $1\in P_{3k+3}$.
        
    \end{enumerate}
    Consequently, the Cayley graph is triangle-free. As in the proof of Theorem \ref{hammingRate}, we have an inclusion
    \begin{equation*}
    \begin{aligned}
        \ann{f_k} \supseteq &\left(\ll x_{3k} + 1\rr +\ll x_{3k+2} +1\rr +\ll x_{3k+3}+1\rr\right)\cdot \left(\ll x_1 + 1\rr +\cdots +\ll x_k +1\rr \right) \\
         &\cdot\left(\ll x_{k+1} + 1\rr +\cdots +\ll x_{2k} +1\rr \right)\cdot\left(\ll x_{2k+1} + 1\rr +\cdots +\ll x_{3k} +1\rr \right). 
    \end{aligned}
    \end{equation*}
    Applying Properties \ref{Property2} and \ref{Property3}, we then obtain $R\left(\ann{f_k}\right) \geq \frac{7}{8}\cdot\left(1-\frac{1}{2^k}\right)^3$.
    
    Note that in the expansion of $f_k+1$, there is only one term containing $x_{3k+1}x_{3k+2}x_{3k+3}$, namely $x_{3k+1}x_{3k+2}x_{3k+3}$ itself. So by Theorem \ref{NeceCon}, $R\left(\ann{f_k}\right) \leq 7/8$. Hence this is a family of rate $7/8$.
    
\end{proof}

\noindent\textbf{Remark.}
The main idea of this construction is to use the long terms to cancel out the even-degree monomials in $(x_{3k+1}+1)(x_{3k+2}+1)(x_{3k+3}+1)$ so that the remaining terms are all of odd degree and thus the product of any three of them cannot be $1$. We can push this idea even further.

\begin{theorem}[The generalized Hamming family]
    \label{FinalConstruction}
    Let $r,k$ be positive integers and $m=2^{r-1}-1$. Then define $\scrC_{r,k} = \ann{f_{r,k}}\subseteq P_{mk+r}$, where 
    $$ f_{r,k} =(x_{mk+1}+1)(x_{mk+2}+1)\cdots (x_{mk+r}+1) +\sum_{i =0}^{m-1}h_i \prod_{j=1}^{k}(x_{ik+j}+1),$$
    and 
    $$\{h_0,h_1,\ldots,h_{m-1}\}=\left\{x_{mk+1}^{s_1}x_{mk+2}^{s_2}\cdots x_{mk+r}^{s_r} \mid (0,0,\ldots ,0)\ne (s_1,s_2,\ldots ,s_r)\in \{0,1\}^r\ and\ \sum_{i=1}^rs_i\ is\ even\right\}.$$
    
   \noindent This is a family of storage codes on triangle-free graphs, which has rate $1-\frac{1}{2^r}$ when $k$ approaches infinity.
\end{theorem}

\begin{proof}
    Similar arguments to those in the proof of Theorem \ref{constrution} can show that the Cayley graph is triangle-free. Again by Theorem \ref{NeceCon}, the presence of the term $y_1y_2\cdots y_r$ implies $R(\scrC_{r,k})\leq 1-\frac{1}{2^r}$. On the other hand, the inclusion
    $$  \ann{f_{r,k}} \supseteq (\ll x_{mk+1} + 1\rr +\ll x_{mk+2} +1 \rr +\cdots + \ll x_{mk+r}+1\rr) \cdot \prod_{i=0}^{m-1}{\sum_{j=1}^{k}{\ll x_{ik+j}+1}\rr}$$
    implies that $R(\ann{f_{r,k}})\geq (1-\frac{1}{2^r})\cdot\left(1-\frac{1}{2^k}\right)^m.$
    As $k$ goes to infinity, $R\left(\scrC_{r,k}\right)$ goes to $1-\frac{1}{2^r}$. Hence this family achieves rate arbitrarily close to one if only we choose a large enough $r$ and $k$.
\end{proof}

\noindent\textbf{The Sparsity of the Constructed Graphs}
\vspace{0.1in}

Consider a family of simple connected graphs $\Gamma(V,E)$, where $V$ is the vertex set and $E$ the edge set. Let $N=|V|$. Suppose $|E|$ is a function of $N$. Note that $|E|\geq N-1$ as $\Gamma$ is connected and $|E|\leq N(N-1)/2$ as $\Gamma$ is simple; hence $|E|=O(N^t)$, where $1\leq t\leq 2$.

We want to figure out the number of edges of the Cayley graphs constructed in Theorem \ref{FinalConstruction}. Let $r$ be a fixed positive integer. Let $n=r+k(2^{r-1}-1)$ and $S$ be the subset of $\F^n$ corresponding to $f_{r,k}+1$. Then the ambient graph of the storage code $\scrC$ is $\Gamma=\Cay{\F^n}{S}$. So the number of vertices is $N = 2^n = 2^r\times 2^{k(2^{r-1}-1)}$ and thus
\begin{equation}
    \label{SparsityFormula}
    k = \frac{\log_2{N}-r}{2^{r-1}-1}.
\end{equation}
Note that $\Gamma$ is a regular graph and each vertex has degree $|S|=2^{r-1}+ (2^{r-1}-1) 2^k$. Using (\ref{SparsityFormula}), we obtain
\begin{equation*}
\begin{aligned}
    |S| &=  2^{r-1}+ (2^{r-1}-1) 2^\frac{\log_2{N}-r}{2^{r-1}-1}  \\
        &= 2^{r-1}+ (2^{r-1}-1) \left(\frac{N}{2^r}\right)^\frac{1}{2^{r-1}-1}.
\end{aligned}
\end{equation*}
Let $E$ be the edges set of $\Gamma$. Then the size of $E$ is 
\begin{equation*}
\begin{aligned}
    |E|&=\frac{N |S|}{2} = \frac{N \left(2^{r-1}+ (2^{r-1}-1) \left(\frac{N}{2^r}\right)^\frac{1}{2^{r-1}-1}   \right)}{2} \\
       &\approx \frac{2^{r-1}-1}{2^{1-{r}/(2^{r-1}-1)}} N^{1+\frac{1}{2^{r-1}-1}}
       = O(N^{1+\frac{1}{2^{r-1}-1}}).
\end{aligned}
\end{equation*}
So the magnitude of $|E|$ is $N^{1+\frac{1}{2^{r-1}-1}}$, the exponent, $1+\frac{1}{2^{r-1}-1}$, approaches $1$ as $r$ goes to infinity. Consequently, the graph family is actually quite sparse.

\section{Rates of some classical ideals}
\label{SectionProof}

Recall that in Section \ref{ConstructionSection} we constructed a family of storage codes, the rate of that family is computed mainly by using Properties \ref{Property2} and \ref{Property3}. We will prove these properties in this section.

Following the convention before, the rate of an ideal is the ratio of the dimension of the ideal to that of its ambient ring when both the ring and the ideal are viewed as vector spaces. To indicate the variables that we are using, 
we usually use $P[x_1,\ldots,x_n]$ as the abbreviation of $P_n={\F \left[x_1,x_2,\ldots ,x_n\right]}/{\langle x_1^2-1,x_2^2-1,\ldots ,x_n^2-1\rangle}$.

Note that $P_n$ is a polynomial quotient ring. So each element in $P_n$ can be represented by a polynomial on some variables. Assume two elements $f,g \in P_n$. If $f=f(x_1,\ldots,x_k)$ and $g=g(x_{k+1},\ldots,x_n)$, then we say $f$ and $g$ are on $\textit{disjoint variables}$.

\begin{lemma}
    \label{TwoBases}
    Viewing $P_n$ as a vector space over $\F$, we have two bases of $P_n$: 
    \begin{itemize}
        \item $B_1 = \left\{x_1^{s_1}\cdots x_n^{s_n}\,\vert\,s_i\in \left\{0,1\right\}\right\} ;$ 
        \item $B_2 = \left\{(x_1+1)^{s_1}\cdots (x_n+1)^{s_n}\,\vert\,s_i\in \left\{0,1\right\}\right\} .$ 
    \end{itemize}
\end{lemma}

\begin{proof}
    In the group algebra $P_n$, the first basis is natural as they are exactly the group elements. Note that $B_2$ has the same size as $B_1$. So it suffices to show that every element in $B_1$ can be linearly expressed by $B_2$. This is true, for example, $x_1 = (x_1+1)+1,\,x_1x_2 = (x_1+1)(x_2+1)+(x_1+1)+(x_2+1)+1$.
\end{proof}

\begin{lemma}
    \label{FunLemma}
    If $f\neq 0$ and $g\neq 0$ are on disjoint variables, then $fg\neq 0$.
\end{lemma}

\begin{proof}
    Assume the ambient ring is $P[x_1,\ldots,x_n]$ and 
    \[f=\sum{a_{s_1,\ldots,s_k} x_1^{s_1}\cdots x_k^{s_k}},g =\sum{b_{s_{k+1},\ldots,s_n}x_{k+1}^{s_{k+1}}\cdots x_n^{s_n}},\]
    where the coefficients $a_{s_1,\ldots,s_k},b_{s_{k+1},\ldots,s_n} \in \F$ are not all zero, respectively. There must be some product $a_{s_1,\ldots,s_k}b_{s_{k+1},\ldots,s_n} \neq 0$. Thus $fg = \sum{a_{s_1,\ldots,s_k}b_{s_{k+1},\ldots,s_n}x_1^{s_1}\cdots x_k^{s_k}x_{k+1}^{s_{k+1}}\cdots x_n^{s_n}\cdots x_n^{s_n}} \neq 0$ since $\left\{x_1^{s_1}\cdots x_n^{s_n}\,\vert\,s_i\in \left\{0,1\right\}\right\} $ are linearly independent over $\F$.
\end{proof}

\begin{lemma}
    \label{Annhilator}
    $\ann{x_i+1} = \ll x_i+1 \rr$, for all $i=1,\ldots,n$.
\end{lemma}

\begin{proof}
    We will only consider the case where $i=1$. The other cases can be done in the same way. The inclusion ``$\supseteq$" is obvious. For $f\in P_n$, we can write $f = \sum{a_{s_1,\ldots,s_n}(x_1+1)^{s_1}\cdots (x_n+1)^{s_n}}$ by Lemma \ref{TwoBases}. If $f\in \ann{x_1+1}$, then
    \begin{equation*}
    \begin{aligned}
        (x_1+1)f &= (x_1+1)\sum{a_{s_1,\ldots,s_n}(x_1+1)^{s_1}(x_2+1)^{s_2}\cdots (x_n+1)^{s_n}} \\
                 &= (x_1+1)\sum_{s_1=0}{a_{s_1,\ldots,s_n}(x_1+1)^{s_1}(x_2+1)^{s_2}\cdots (x_n+1)^{s_n}}\\
                 &= (x_1+1)\sum{a_{0,s_2,\ldots,s_n}(x_2+1)^{s_2}\cdots (x_n+1)^{s_n}}=0.
    \end{aligned}
    \end{equation*}
    By Lemma \ref{FunLemma} it follows that $\sum{a_{0,s_2,\ldots,s_n}(x_2+1)^{s_2}\cdots (x_n+1)^{s_n}}=0$. By Lemma \ref{TwoBases} each coefficient $a_{0,s_2,\ldots,s_n}=0$. Hence $f = \sum{a_{1,s_2,\ldots,s_n}(x_1+1)(x_2+1)^{s_2}\cdots (x_n+1)^{s_n}}$, and thus $f\in \ll x_1+1\rr$.
     
\end{proof}

\begin{proposition}
    \label{SimpleCase}
    The rate of the ideal $ \ll x_1+1\rr\ll x_2+1\rr\cdots \ll x_k+1\rr $ is $ {1}/{2^k}$, for $1\leq k\leq n$.
\end{proposition}

\begin{proof}
    Let $f\in P_n$, say $f = \sum{a_{s_1,\ldots,s_n}(x_1+1)^{s_1}\cdots(x_n+1)^{s_n}}.$
    If $f \in \ll x_1+1\rr\ll x_2+1\rr\cdots \ll x_k+1\rr$, then $f \in \ann{x_i+1}$ and thus
    $a_{s_1,\ldots,s_{i-1},0,s_{i+1},\ldots,s_n} = 0$ as in the proof of Lemma \ref{Annhilator} for $i = 1,\ldots,k$.
    Then $f = \sum{a_{1,\ldots,1,s_{k+1},\ldots,s_n}(x_1+1)\cdots(x_k+1)(x_{k+1}+1)^{s_{k+1}}\cdots(x_n+1)^{s_n}}$.
    On the other hand, every such $f$ lies in  $ \ll x_1+1\rr\ll x_2+1\rr\cdots \ll x_k+1\rr $. So the ideal has the following basis 
    $$\left\{(x_1+1)\cdots(x_k+1)(x_{k+1}+1)^{s_{k+1}}\cdots (x_n+1)^{s_n}\ |\ s_i\in \left\{0,1\right\}\right\}.$$
    The dimension of the ideal is $2^{n-k}$ and thus the rate is ${2^{n-k}}/{2^n} = {1}/{2^k}$.
    
\end{proof}

\begin{proposition}
    \label{CapMul}
    $\ll x_1+1 \rr\cdots \ll x_k+1 \rr\cap \ll x_{k+1}+1\rr = \ll x_1+1 \rr\cdots\ll x_k+1 \rr \ll x_{k+1}+1\rr$, for $1\leq k \leq n-1$.
\end{proposition}

\begin{proof}
    The inclusion ``$\supseteq$'' is obvious. Let $f \in \ll x_1+1 \rr\cdots \ll x_k+1 \rr\cap \ll x_{k+1}+1\rr$. Then we can write $f = (x_1+1)\cdots(x_k+1)g$ for some $g \in P_n$. Applying the fact that $(x_i+1)x_i=x_i+1$, we have $(x_1+1)\cdots(x_k+1)g = (x_1+1)\cdots(x_k+1)g(x_1=1,\cdots,x_k=1)$, so without loss of generality we can assume $g$ is only on variables $x_{k+1},\ldots,x_n$. As $f \in \ll x_{k+1}+1\rr$, we have $(x_{k+1}+1)f = (x_1+1)\cdots(x_k+1)(x_{k+1}+1)g = 0$. Hence by Lemma \ref{FunLemma} $(x_{k+1}+1)g = 0$ and by Lemma \ref{Annhilator} $g\in \ll x_{k+1}+1 \rr$. We are done.
\end{proof}

\begin{theorem}
    \label{RateAdd}
    The rate of the ideal $\ll x_1+1\rr+\ll x_2+1 \rr+\cdots+\ll x_k+1 \rr$ is $1-{1}/{2^k}$, for $1\leq k\leq n$.
\end{theorem}

\begin{proof}
    Recall that for two vector subspaces $V$ and $W$ of $P_n$, we have $\dim (V+W) = \dim V + \dim W -\dim V\cap W$. Applying this formula, we have 
    \begin{equation*}
    \begin{aligned}
        \dim \sum_{i=1}^k\ll x_i+1\rr &= \sum_{s=1}^k{(-1)^{s+1}{\sum_{i_1<\cdots<i_s}}{\dim \ll x_{i_1}+1\rr\cap\cdots\cap\ll x_{i_s}+1 \rr}}\\
        &=\sum_{s=1}^k{(-1)^{s+1}\sum_{i_1<\cdots<i_s}}{\dim \ll x_{i_1}+1\rr\cdots\ll x_{i_s}+1 \rr},
    \end{aligned}
    \end{equation*}
    where we have used Proposition \ref{CapMul} in the last equality. Dividing both sides by the dimension of the ambient ring and applying Proposition \ref{SimpleCase}, we obtain
    \[R\left(\sum_{i=1}^k\ll x_i+1\rr\right)=\sum_{s=1}^k(-1)^{s+1}\binom{k}{s}\frac{1}{2^s}=1-\frac{1}{2^k}.\]
\end{proof}

\begin{lemma}
    \label{ProductLI}
    Let $\left\{f_i(x_1,\ldots,x_r)\right\}_{i = 1,\ldots,k},\left\{g_j(y_1,\ldots,y_s)\right\}_{j = 1,\ldots,\ell} $ be two sets of elements on disjoint variables in $ P[x_1,\ldots,x_r,y_1,\ldots,y_s]$. If these are two linearly independent sets, then the set $\left\{f_i g_j\,\vert\,i = 1,\ldots,k,j = 1,\ldots,\ell\right\}$ is linearly independent.
\end{lemma}

\begin{proof}
    For convenience, we denote 
    \[\left\{v_1,\ldots,v_n\right\} = \left\{x_1^{t_1}\cdots x_r^{t_r}\,\vert \, t_1,\ldots,t_r \in \left\{0,1\right\}\right\}, \]
    \[\left\{w_1,\ldots,w_m\right\} = \left\{y_1^{t_1}\cdots y_s^{t_s}\,\vert \, t_1,\ldots,t_s \in \left\{0,1\right\}\right\}.\]
    Then for $i=1,\ldots,k$ and $j=1,\ldots,\ell$, we have $f_i = \sum_{c=1}^n{a_{ic}v_c}$ and $g_j = \sum_{e=1}^m{b_{je}w_e}$, where $a_{ic},b_{je} \in \F$.
    Assume a zero linear combination $\sum_{i,j}{\lambda_{ij}f_ig_j}=0$ for some $\lambda_{ij} \in \F$. Then
    \[ 0 = \sum_{i,j}{\lambda_{ij}f_ig_j}=\sum_{i,j}{\lambda_{ij} \sum_{c=1}^n{a_{ic}v_c} \sum_{e=1}^m{b_{je}w_e}}= \sum_{c,e}{\sum_{i,j}{\lambda_{ij}a_{ic}b_{je}\cdot v_c w_e}}.\]
    Note that $\left\{v_c w_e \,\vert \, c=1,\ldots,n,e = 1,\ldots,m\right\}$ is the first basis in Lemma \ref{TwoBases}, so we have 
    \[\sum_{i,j}{\lambda_{ij}a_{ic}b_{je}} = \sum_i{a_{ic}\sum_j{\lambda_{ij}b_{je}}} = 0,\  c=1,2,\ldots,n,\ e=1,2,\ldots,m.\]
    Since $\left\{f_i\right\}$ are linearly independent, the coefficient vectors $\left\{[a_{i1}\ a_{i2}\ \cdots\ a_{in}]\ |\ i=1,\ldots,k\right\}$ are also linearly independent. Then $\sum_j{\lambda_{ij}b_{je}} = 0$ for all $i$ and $e$, which further implies $\lambda_{ij}=0$ for all $i$ and $j$ as $\left\{g_j\right\}$ are linearly independent. We are done.
\end{proof}

\begin{theorem}
    \label{RateMul}
    Let $\left\{f_i(x_1,\ldots,x_r)\right\}_{i = 1,\ldots,k},\left\{g_j(y_1,\ldots,y_s)\right\}_{j = 1,\ldots,\ell} $ be two sets of elements on disjoint variables in $ P[x_1,\ldots,x_r,y_1,\ldots,y_s]$. Then we have a rate formula $R(\alpha\beta) = R(\alpha) R(\beta)$, where $\alpha =  \ll f_1,\ldots,f_k\rr$ and $ \beta =\ll g_1 ,\ldots,g_\ell\rr$.
\end{theorem}

\begin{proof}
    Let $ V= \alpha \cap P[x_1,\ldots,x_r]$ and $ W =  \beta\cap P[y_1,\ldots,y_s]$.
    Then $V$ and $W$ are vector spaces and we assume $\left\{v_1,v_2,\ldots,v_n\right\}$ and $\left\{w_1,w_2,\ldots,w_m\right\}$ are their bases, respectively. Define the tensor product $V\otimes W$ to be the $\F$-span of $\left\{v_cw_e\,\vert\,c = 1,\ldots,n,e = 1,\ldots,m\right\}$. By Lemma \ref{ProductLI} $\left\{v_cw_e\right\}$ are linearly independent and
    thus $\dim V\otimes W = \dim V\times \dim W =  n\times m$.
    
    We claim that $\alpha\beta = V\otimes W$. The inclusion ``$\supseteq$" is clear. Let $h\in \alpha\beta$, say $h = \sum_{ij}d_{ij}\cdot f_ig_j$, where $d_{ij}\in P_{r+s}$. 
    For each $d_{ij}$, we can write $d_{ij} = \sum_k{a_{ijk}}b_{ijk}$, where $a_{ijk}\in P[x_1,\ldots,x_r],b_{ijk}\in P[y_1,\cdots,y_s]$. Hence 
    \[
    h = \sum_{ij}\sum_k{a_{ijk}}b_{ijk}\cdot f_ig_j = \sum_k\sum_{ij}a_{ijk}f_i\cdot b_{ijk}g_j.
    \]
    Note that $a_{ijk}f_i \in V$ and $b_{ijk}g_j\in W$. Then $a_{ijk}f_i\cdot b_{ijk}g_j \in V\otimes W$ and so is $h$, thus $\alpha\beta = V\otimes W$. Hence $\dim \alpha\beta = n\times m$.

    Similarly, we have $\alpha = V\otimes P[y_1,\ldots,y_s]$ and thus $\dim \alpha = \dim V\times\dim P[y_1,\ldots,y_s] = n\times 2^s$, obtain $R(\alpha) = {\dim \alpha}/{2^{r+s}} = {n}/{2^r}$. Similarly, $R(\beta)={m}/{2^s}$.
    Finally, we obtain  
    $$ R(\alpha\beta) = \frac{\dim \alpha\beta}{2^{r+s}}=\frac{n\times m}{2^r\times2^s}=R(\alpha) R(\beta).$$
        
\end{proof}

\begin{corollary}
    \label{corollary}
    $R \left(\prod_{i = 1}^{n}{\sum_{j = 1}^{k_i}{\ll x_{ij}+1\rr}}\right) 
    =\prod_{i = 1}^{n}{R\left(\sum_{j = 1}^{k_i}{\ll x_{ij}+1\rr}\right)}$.
\end{corollary}

\begin{proof}
    This follows from Theorem \ref{RateMul} by induction.
\end{proof}

\vspace{0.2in}
\noindent{\bf Acknowledgements.} While we are finishing writing up this paper, we became aware of \cite{bargmsch}, in which the authors construct a family of storage codes of rate asymptotically one. Since both the construction in the current paper and the construction in \cite{bargmsch} are generalizing the Hamming family, they look similar. However the methods used to compute the rates of storage codes are very different.


{\small
Department of Mathematics and National Center for Applied Mathematics Shenzhen, Southern University of Science and Technology, Shenzhen 518055, China 

\text{Email address:} {\tt hhxiang1999@foxmail.com}

\vspace{0.2cm}
Department of Mathematics and Shenzhen International Center of Mathematics, Southern University of Science and Technology, Shenzhen 518055, China

\text{Email address:} {\tt xiangq@sustech.edu.cn}
}

\end{document}